\documentclass[oneside, 10pt]{amsart}

\usepackage[letterpaper,left=1.2in,right=1.2in,top=1in,bottom=1in]{geometry}

\usepackage[english]{babel}

\usepackage{amsmath, amssymb, amsthm}
\usepackage{amsrefs}

\usepackage[all]{xy}

\usepackage{hyperref}
\hypersetup{colorlinks=true,linkcolor=blue,anchorcolor=blue,citecolor=blue}

\newtheorem{theoremintro}{Theorem}
\newtheorem{questionintro}[theoremintro]{Question}

\newtheorem{proposition}{Proposition}
\newtheorem{theorem}[proposition]{Theorem}
\newtheorem{lemma}[proposition]{Lemma}

\newtheorem{corollary}[proposition]{Corollary}

\newcommand{\topo}{{\mathrm{top}}}
\newcommand{\hofib}{\operatorname{hofib}}
\newcommand{\GL}{\operatorname{GL}}
\newcommand{\SL}{\operatorname{SL}}
\newcommand{\Gm}{\mathbb{G}_m}
\newcommand{\PGL}{\operatorname{PGL}}

\newcommand{\Hoh}{\mathrm{H}}
\newcommand{\Eoh}{\mathrm{E}}
\newcommand{\Br}{\operatorname{Br}}
\newcommand{\ind}{\operatorname{ind}}
\newcommand{\per}{\operatorname{per}}
\providecommand{\deg}{\operatorname{deg}}

\newcommand{\iso}{\cong}
\newcommand{\ZZ}{\mathbb{Z}}
\newcommand{\tensor}{\otimes}
\newcommand{\id}{\mathrm{id}}

\newcommand{\End}{\operatorname{End}}

\newcommand{\Spec}{\operatorname{Spec}}
\newcommand{\Mat}{\operatorname{Mat}}
\newcommand{\et}{\textit{\'et}}
\newcommand{\fpqc}{\textit{fpqc}}
\newcommand{\B}{\mathrm{B}}

\newcommand{\NN}{\mathbb{N}}
\newcommand{\QQ}{\mathbb{Q}}
\newcommand{\CC}{\mathbb{C}}
\newcommand{\RR}{\mathbb{R}}
\newcommand{\tors}{\mathrm{tors}}

\newcommand{\Sq}{\mathrm{Sq}}

\newcommand{\Sp}{\mathrm{Sp}}
\newcommand{\PSp}{\mathrm{PSp}}
\newcommand{\PGSp}{\mathrm{PGSp}}
\newcommand{\Og}{\mathrm{O}}
\newcommand{\PO}{\mathrm{PO}}
\newcommand{\SOg}{\mathrm{SO}}
\newcommand{\SU}{\operatorname{SU}}
\newcommand{\PGO}{\mathrm{PGO}}

\newcommand{\extended}[1]{{\color{ex}}}

\newcommand{\remove}[1]{}

\newcommand{\tr}{{\mathrm{t}}}
\newcommand{\syp}{{\mathrm{sp}}}
\newcommand{\trdeg}{\mathrm{trdeg}}

\newcommand{\smallSMatII}[4]{\left[\begin{smallmatrix} {#1} & {#2} \\ {#3} &
{#4} \end{smallmatrix}\right]}

\newcommand{\calL}{{\mathcal{L}}}

\newcommand{\calO}{{\mathcal{O}}}
\newcommand{\calP}{{\mathcal{P}}}

\newcommand{\etrm}{\mathrm{\acute{e}t}}

\newcommand{\exterior}{{\textstyle \bigwedge}}
\newcommand{\EEnd}{\mathcal{E}\!\mathit{nd}}
\newcommand{\HHom}{\mathcal{H}\!\mathit{om}}
\newcommand{\Pic}{\mathrm{Pic}}
\newcommand{\PP}{\mathbb{P}}

\newcommand{\ed}{\operatorname{ed}}

\newtheoremstyle{roman} 
 {8.0pt plus 2.0pt minus 4.0pt}  
 {8.0pt plus 2.0pt minus 4.0pt}  
 {\normalfont} 
 {}  
 {\bfseries}  
 {.}  
 {5pt plus 1pt minus 1pt} 
 {} 

\theoremstyle{roman}

\newtheorem{example}[proposition]{Example}
\newtheorem{remark}[proposition]{Remark}

\theoremstyle{plain}

\title{Azumaya algebras without involution}

\author{Asher Auel}
\author{Uriya A.\ First}
\author{Ben Williams}
\address{Asher Auel, Department of Mathematics, Yale University, 10 Hillhouse
Avenue, New Haven, CT 06511, United Stated}
\address{Uriya A.\ First and Ben Williams, Department of Mathematics,
The University of British Columbia, Room 121, 1984 Mathematics Road,
Vancouver, B.C., Canada V6T 1Z2}

\date{\today}


\begin{document}

\begin{abstract}
Generalizing a theorem of Albert, Saltman showed that an Azumaya
algebra $A$ over a ring represents a $2$-torsion class in the Brauer
group if and only if there is an algebra $A'$ in the Brauer class of
$A$ admitting an involution of the first kind. Knus, Parimala, and
Srinivas later showed that one can choose $A'$ such that $\deg
A'=2\deg A$. We show that $2\deg A$ is the lowest degree one can
expect in general. Specifically, we construct an Azumaya algebra $A$
of degree $4$ and period $2$ such that the degree of any algebra $A'$
in the Brauer class of $A$ admitting an involution is divisible by
$8$.

 Separately, we provide examples of split and non-split Azumaya algebras of degree $2$ admitting symplectic
 involutions, but no orthogonal involutions. These stand in contrast to the case of central simple algebras of even
 degree over fields, where the presence of a symplectic involution implies the existence of an orthogonal involution and vice
 versa.
\end{abstract}

\maketitle

\section*{Introduction}


Let $A$ be a central simple algebra over a field $F$.  It is a
classical result of Albert \cite{Al61StructureOfAlgs}*{X~\S9~Thm.~19}
(cf.\ \cite{InvBook}*{Thm.~3.1(1)}) that $A$ has an involution of the
first kind if and only if the Brauer class $[A]$ has period at most $2$, i.e.,
lies in the $2$-torsion part of the Brauer group $\Br(F)$. This
characterization was later extended and clarified by Scharlau
\cite{scharlau:involutions} and in unpublished work by Tamagawa.

Albert's theorem does not generally extend to Azumaya algebras $A$
over a commutative ring $R$.  However, Saltman~\cite[Thm.~3.1(a)]{Sa78}
showed that a class $[A] \in \Br(R)$ has period dividing $2$ if and only if
there is an Azumaya algebra $A'$ in the class $[A]$ admitting an
involution of the first kind. Furthermore, one can take $A'=A$ when
$R$ is semilocal or when $\deg A=2$, see \cite[Thms.~4.1,
4.4]{Sa78}. A later proof given by Knus, Parimala, and
Srinivas~\cite[Thm.~4.1]{KnParSri90}, which applies in the generality of
Azumaya algebras over schemes, shows that in this context $A'$ can be
choosen so that $\deg A'=2\deg A$.

\medskip

Suppose $X=\Spec R$ is connected and $[A]$ is a $2$-torsion class in
$\Br(X)$. It is natural to ask, in the context of the results of
\cite{Sa78} and \cite{KnParSri90}, whether in general $2\deg A$ is the
least possible degree of a representative $A'$ in the class $[A]$
admitting an involution of the first kind.  We answer this question in
the affirmative.

\begin{theoremintro}\label{TH:main}
There exists a commutative domain $R$ and an Azumaya $R$-algebra $A$
satisfying 
\[ 
\deg A=4,\qquad \per A=2 ,
\]
and such that any representative $A'$ in the Brauer class $[A]$ admitting an
involution of the first kind satisfies 
\[
{8\mid\deg A'}.
\]
\end{theoremintro}

This construction is done using approximations of the universal bundle
over the classifying stack of $\SL_4/\mu_2$, similarly to the methods
of \cite{antieau2014-a}.

\medskip

Involutions of the first kind of central simple algebras may be
further classified as being orthogonal or symplectic, see
\cite[\S2]{InvBook}. This classification extends to Azumaya algebras
over connected schemes \cite[\S{}III.8]{Kn91}.

It is well known that a central simple algebra of even degree over a
field admits either involutions of both types or of neither type.  We
show that this is not the case for Azumaya algebras. We construct both
split and non-split Azumaya algebras of degree $2$, each admitting a
symplectic involution but no orthogonal involutions. The split examples
are constructed using quadratic spaces and Clifford algebras, whereas
the non-split example relies on generic methods to construct an
Azumaya algebra without zero divisors that specializes to the given
split example. In fact, the latter example can made to specialize to
any prescribed Azumaya algebra with involution over an affine scheme.

\medskip

We remark that simply constructing Azumaya $R$-algebras $A$ without
involution and with period $2$ is a nontrivial problem. For instance,
it is impossible when $R$ is a semilocal ring (as mentioned above).
In addition, 
for many low-dimensional rings, e.g., integer rings of global fields,
the bound $\deg A'\leq 2\deg A$ of
\cite{KnParSri90} is not sharp; see \cite[Rem.~11.5]{Fi15}.

In \cite[\S4]{Sa78}, Saltman provided a criterion for split Azumaya
algebras to have an involution, which was used to provide explicit
split examples.  These methods were extended in
\cite[\S11]{Fi15} to give non-split examples.  Alternatively, given a
split example, a sufficiently generic algebra of period 2, which
specializes to the split example, cannot have any involution.  The
major difficultly then lies, in trying to fully verify
Theorem~\ref{TH:main}, with showing that no Brauer equivalent algebra
of the same degree can have an involution.  This is where the
topological obstruction theory is useful.

\medskip

 Given the above results, we conjecture that:
 \begin{enumerate}
 \item[(a)] For any $n\geq 2$, there is an Azumaya algebra $A$
 of degree $2^n$ and period $2$ such any $A'\in [A]$ with an involution of the first kind
 satisfies $2^{n+1}\mid \deg A'$.
 \item[(b)] For any $n\geq 1$ and $n\geq m\geq 0$, there is an Azumaya algebra $A$
 with $\deg A =2^n$ and $\ind A=2^m$
 which admits a symplectic involution, but not an orthogonal involution.
 \item[(c)] For any $n\geq 2$ and $n\geq m\geq 0$, there is an Azumaya algebra $A$
 with $\deg A =2^n$ and $\ind A=2^m$
 which admits an orthogonal involution, but not a symplectic involution.
 \end{enumerate}
 We have settled (a) when $n=2$ and (b) when $n=1$ and $m\in\{0,1\}$.

 We would also like to mention the case of involutions of the second
 kind, which we do not take up in this article.  Let $K/F$ be a
 separable quadratic extension of fields.  The Albert--Riehm--Scharlau
 Theorem \cite{InvBook}*{Thm.~3.1(2)} asserts that a central simple
 $K$-algebra $A$ has an involution of the second kind with fixed field
 $F$ if and only if the corestriction $\mathrm{cor}_{K/F}[A]$ is
 trivial in $\Br(F)$.  Saltman~\cite[Thm.~3.1(b)]{Sa78} showed that a
 class $[A] \in \Br(S)$, where $S/R$ is a quadratic Galois extension
 of commutative rings, has trivial corestriction if and only if there
 is an Azumaya algebra $A'$ in the class $[A]$ admitting an involution
 of the second kind fixing $R$ pointwise.  A different proof given by
 Knus, Parimala, and Srinivas~\cite[Thm.~4.2]{KnParSri90}, shows that
 $A'$ can be choosen so that $\deg A'=2\deg A$.  It is then natural to
 ask, whether in general $2\deg A$ is the least possible degree of a
 representative $A'$ in the class $[A]$ admitting an involution of the
 second kind.

\begin{questionintro}\label{Q:second}
Do there exist a commutative domain $R$, a  quadratic Galois 
extension $S/R$, and an Azumaya $S$-algebra $A$ of degree $n$ with
$\mathrm{cor}_{S/R}[A]$ trivial, such that any representative $A'$ in
the Brauer class $[A]$ admitting an involution of the second kind with
fixed ring $R$ has $2n\mid \deg A'$?
\end{questionintro}

\medskip

We warmly thank B.\ Antieau, without whom this paper would not have
been written in its present form. We thank R.\ Parimala and D.\
Saltman for useful remarks and suggestions.  The second named author
is grateful to Z.\ Reichstein for many beneficial discussions.  The
first author was partially sponsored by NSF grant DMS-0903039 and a
Young Investigator grant from the NSA.

\medskip

The paper is organized as follows: In the first section, we provide
some of the necessary background on Azumaya algebras and their
involutions. The second section is topological in content.  After a
brief primer on the required theory of classifying spaces and
cohomology, it is dedicated to giving certain obstructions to the
existence of maps between classifying spaces. These obstructions are
described in Proposition~\ref{something}, which is used in Section~3
to construct an example of an Azumaya algebra of period 2 and degree 4
for which any Brauer-equivalent algebra with involution has degree no
less than $8$. Section 4 exhibits a split Azumaya algebra equipped
with a symplectic but not an orthogonal involution. The fifth and
final section is concerned with two constructions of non-split Azumaya
algebras that specialize to the split example of the previous section,
thus consequently giving examples of non-split Azumaya algebras
admitting symplectic but not orthogonal involutions.

\bigskip

\section{Preliminaries}
\label{sec:preliminaries}


We assume throughout that $2$ is invertible in all rings and schemes.

\medskip

We begin by recalling some facts about Azumaya algebras and their
involutions. For simplicity, we have restricted here to connected
rings, schemes, etc., but the extension to disconnected settings
requires only some additional bookkeeping.

\subsection{Azumaya algebras}

Let $R$ be a connected commutative ring. Recall that an $R$-algebra
$A$ is called \emph{Azumaya} if $A$ is locally free of finite nonzero rank and
the map $\Phi:A\otimes_R A^{\mathrm{op}}\to \End_R(A)$ given by
$\Phi(a\otimes b^\mathrm{op})(x)=axb$ is an isomorphism.
Equivalently, $A$ is Azumaya if there exists a faithfully flat
\'{e}tale $R$-algebra $S$ such that $A\otimes_RS\cong \Mat_{n\times
n}(S)$ as $S$-algebras, \cite[Thm.~III.5.1.1]{Kn91}. The number $n$ is
called the \emph{degree} of $n$ and is denoted $\deg A$.  When $R$ is
a field, Azumaya algebras are precisely the central simple algebras.

 Two Azumaya algebras $A$, $B$ over $R$ are \emph{Brauer equivalent} if
 there are locally free $R$-modules $P$ and $Q$ of finite rank such that
 $A\otimes_R\End_R(P)\cong B\otimes_R\End_R(Q)$ as $R$-algebras.
 The equivalence class of $A$ is called the \emph{Brauer class} of $A$ and denoted $[A]$.
 The \emph{index} of $A$, denoted $\ind A$, is $\mathrm{gcd}\{\deg A'\,|\,A'\in [A]\}$.
 Unlike in the case where $R$ is a field, it is possible that there might be no representative
 $A'$ in the class $[A]$ with $\deg A'=\ind A$; see \cite{antieau2014-a}.
 The collection of Brauer
 classes with the operation $[A]\otimes [B]:=[A\otimes_RB]$
 forms a group, denoted $\Br(R)$,
 and called the \emph{Brauer group} of $R$.
 See \cite{DeMeyIngr71SeparableAlgebras}, \cite[\S{}III.5]{Kn91} for further details.

Let $X$ be a connected scheme.  An Azumaya algebra of degree $n$ over
 $X$ is a sheaf of $\calO_X$-modules that is locally isomorphic in the
 \'{e}tale topology to $\Mat_{n\times n}(\calO_X)$.  That is, $X$ has
 an \'{e}tale cover (possibly disconnected) $\pi:U\to X$ such that $\pi^*A\cong \Mat_{n\times
 n}(\calO_U)$.  If $X=\Spec R$ for a commutative ring $R$, then $A$ is
 Azumaya over $X$ if and only if $\Gamma(X,A)$ is Azumaya as an
 $R$-algebra.  The degree of $A$, the Brauer class $[A]$, the Brauer
 group of $X$, etc.\ are defined as above---replace $P$ and $Q$ by
 locally free $\calO_X$-modules of finite rank and $\End(-)$ by the
 sheaf $\EEnd(-)$. See \cite[Chp.~IV]{Milne80EtaleCohomology} for
 further details.

\medskip

In the same way, one can define Azumaya algebras, Brauer classes, etc., in any locally ringed topos, \cite{grothendieck1968-a}.
In the case where $X$ is a topological space and $\calO_X$ is
the sheaf of continuous functions into $\CC$, an Azumaya algebra over $X$ is called a \emph{topological Azumaya
 algebra}. A topological Azumaya algebra $A$ over $X$ is therefore a sheaf of $\CC$-algebras over $X$ the fibers
of which are isomorphic to $\Mat_{n\times n}(\CC)$; see \cite{antieau2014}.

\subsection{Involutions}

Let $F$ be a field and let $A$ be a central simple $F$-algebra.
Recall that an involution on $A$ is an additive map $\sigma:A\to A$
such that $\sigma\circ \sigma=\id_{A}$ and $\sigma(ab)=\sigma(b)\cdot
\sigma(a)$ for all $a,b\in A$.  The involution $\sigma$ is of the
\emph{first kind} if it fixes $F$ pointwise and of the \emph{second
kind} otherwise. Henceforth, all involutions are assumed to be of the
first kind. An Azumaya algebra admitting an involution will be termed \textit{involutary}.

Write $n=\deg A$. Recall from \cite[\S2]{InvBook} that the
$F$-dimension of $\{\sigma(a)-a\,|\, a\in A\}$ is either
$\frac{1}{2}n(n-1)$ or $\frac{1}{2}n(n+1)$, in which case the type of
$\sigma$ is said to be \emph{orthogonal} or \emph{symplectic}
respectively. Symplectic involutions can exist only when $n$ is even.
\medskip

Let $X$ be a connected scheme (resp.\ topological space) and let $A$
be an Azuamya algebra (resp.\ topological Azumaya algebra) over
$X$. An involution on $A$ is an $\calO_X$-module isomorphism $\sigma:A\to
A$ such that $\sigma\circ \sigma=\id_{A}$ and the identity
$\sigma_*(ab)=\sigma_* b\cdot \sigma_* a$ holds on sections. In this
case, we say that the pair $(A,\sigma)$ is an \emph{Azumaya algebra
with involution} (of the first kind). If $X$ is $\Spec R$ for a commutative ring
$R$, a map $\sigma$ is an involution if and only if
$\sigma_*:\Gamma(X,A)\to\Gamma(X,A)$ is an involution of Azumaya
$R$-algebras.

Let $x$ be a point of $X$. In the algebraic case, let $k = k(x)$, and
in the topological, let $k= \CC$. By pulling back, $A$ induces a
central simple $k$-algebra $A_x$ and $\sigma$ induces an involution of
the first kind $\sigma_x:A_x\to A_x$, which can be either orthogonal
or symplectic. Since $X$ is connected, the type of $\sigma_x$ is
independent of $x$, so we can simply say that $\sigma$ is orthogonal
or symplectic. See \cite[\S{}III.8]{Kn91} for an extensive discussion
in the case of affine schemes and
\cite[\S{}1.1]{parimala_srinivas:brauer_group_involution} in the case
of arbitrary schemes.

\begin{example}\label{EX:involutions}
Denote by $\tr$ the standard matrix transpose. 
\begin{enumerate} \renewcommand{\itemsep}{3pt}
\item The transpose $\tr$ defines an orthogonal involution of $\Mat_{n\times
n}(\calO_X)$.

\item Suppose $n$ is even and let $\syp$ denote the standard
symplectic involution on $\Mat_{n\times n}(\calO_X)$, defined 
 by
 \[\syp_*\smallSMatII{A}{B}{C}{D}=\smallSMatII{A}{B}{C}{D}^{\syp}:=\smallSMatII{A^\tr}{-C^\tr}{-B^\tr}{D^\tr}\]
 on sections---here $A,B,C,D$ are $\frac{n}{2}\times \frac{n}{2}$
 matrices.

 \item Let $b:\calP\times \calP\to \calL$ be nondegenerate symmetric
 (resp.\ anti-symmetric) bilinear form, where $\calP$ is a locally
 free $\calO_X$-module of finite rank and $\calL$ is an invertible
 $\calO_X$-module on $X$.  Then $b$ induces an orthogonal (resp.\
 symplectic) involution $\sigma$ on $\EEnd(\calP)$; on sections,
 $\sigma_*$ can be recovered from the equality
 $b_*(ax,y)=b_*(x,\sigma_*(a)y)$.  Moreover any orthogonal (resp.\
 symplectic) involution on $\EEnd(\calP)$ is obtained in this manner,
 as $\calL$ varies, by virtue of \cite[Thm.~4.2]{Sa78}.  For a
 generalization see \cite[Thm.~5.7]{Fi14}.
 \end{enumerate}

\end{example}

 \begin{example}\label{EX:deg-two-Az-algebra}
 Azumaya algebras of degree $2$ may always be equipped with a symplectic involution.
 To see this, let $A$ be an Azumaya algebra of degree $2$ over $X$.
 If $U$ is an open affine subscheme of $X$,
 define $\sigma_U:\Gamma(U,A)\to \Gamma(U,A)$ by $\sigma_U(a)=\mathrm{Trd}(a)-a$, where
 $\mathrm{Trd}(a)$ denotes the reduced trace of $a\in \Gamma(U,A)$; see {\cite[Thm.~4.1]{Sa78}}.
 \end{example}

\begin{example}
To construct split Azumaya algebras with no involution of the first
kind, we simply need to find a locally free $\calO_X$-module $\calP$
such that $\calP \not\cong \calP^{\vee}\otimes\calL$ for any
invertible $\calO_X$-module $\calL$ on $X$.  Examples over projective
schemes are easy to construct, e.g., $X=\PP^1$ and $\calP = \calO_X
\oplus \calO_X \oplus \calO_X \oplus \calO_X(1)$.  Similar examples
were constructed over appropriate Dedekind domains in
\cite[\S11]{Fi15}.
\end{example}

 \begin{proposition}\label{PR:local-isom-of-algebras-with-inv}
 Let $X$ be a connected scheme, let $A$ be an Azumaya algebra of
 degree $n$ over $X$, and let $\sigma:A\to A$ be an involution of
 orthogonal or symplectic type.  Then $(A,\sigma)$ is locally
 isomorphic to $(\Mat_{n\times n}(\calO_X),\tr)$ or
 $(\Mat_{n\times n}(\calO_X),\syp)$ in the \'{e}tale topology, respectively.
 \end{proposition}

 \begin{proof}
 We can assume $X$ is affine, and in this case we refer to
 \cite[\S{}III.8.5]{Kn91}.  See also \cite[\S{}1.1]{parimala_srinivas:brauer_group_involution}.
 \end{proof}

\subsection{Cohomology}
\label{subsection:cohomology}

Let $X$ be a scheme and let $X_{\et}$ denote $X$ endowed with the \'{e}tale topology. Recall that $\PGL_n$ is the sheaf
of algebra automorphisms of $\Mat_{n\times n}$. There is a standard bijective correspondance between isomorphism
classes of Azumaya algebras of degree $n$ and of $\PGL_n$-torsors over $X_\et$, and each is classified by
$\Hoh^1_\et(X, \PGL_n)$; see \cite[p.~145]{Kn91} and \cite[Chp.~III, Cor.~4.7]{Milne80EtaleCohomology}. The general
principle of this correspondence appears as \cite[Cor.\ III.2.2.6]{giraud1971}.

Suppose now that $2$ is invertible on $X$ and let $\PGO_{2n}$ and
$\PGSp_{2n}$ denote the sheaf of automorphisms of $(\Mat_{2n\times
2n},\tr)$ and $(\Mat_{2n\times 2n},\syp)$, respectively. In the same
manner as the above, there is a bijective correspondence between
isomorphism classes of $\PGO_{2n}$-torsors (resp.\
$\PGSp_{2n}$-torsors) over $X_{\et}$ and isomorphism classes of
Azumaya algebras with an orthogonal (resp.\ symplectic) involution of
degree $2n$ over $X$, see
Proposition~\ref{PR:local-isom-of-algebras-with-inv},
\cite[\S{}III.8.5]{Kn91}, or
\cite[\S1.1(iii)]{parimala_srinivas:brauer_group_involution}.

Similarly, when $X$ is a topological space, isomorphism classes of
topological Azumaya algebras with an orthogonal (resp.\ symplectic)
involution of degree $2n$ over $X$ correspond to isomorphism classes
of $\PGO_{2n}(\CC)$-bundles (resp.\ $\PGSp_{2n}(\CC)$-bundles) over
$X$. The Skolem--Noether Theorem implies that $\PGO_{2n}(\CC)\cong
\PO_{2n}(\CC)=\Og_{2n}(\CC)/\mu_2$ and $\PGSp_{2n}(\CC)\cong \PSp_{2n}(\CC)=\Sp_{2n}(\CC)/\mu_2$
where $\mu_2=\{\pm 1\}$.

 \section{Homotopy Theory}

 In this section, we use algebraic topology to find obstructions to maps between classifying spaces of Lie groups by
 considering the rational cohomology of maximal tori. This idea appears at least as long ago as \cite{adams1976}. Our
 application requires that we consider spaces that have the same homotopy and homology groups as classifying spaces in a
 certain range, but which differ in general; this contrasts with the results of \cite{jackowski1992} where the classifying space is considered in its entirety.

 In this section and the next, the notation $\Hoh^*(X, A)$ will be used for sheaf cohomology groups with values in the sheaf represented by the topological
 group $A$. When $A$ is discrete, and $X$ has the homotopy type of a CW complex, this coincides with singular cohomology with coefficients in $A$.

 We make extensive use of classifying spaces of topological groups. A thorough discussion of this material may be found in
 \cite [\S 8 and 9]{may1975}, but the discussion in \cite[Chp.~6]{mccleary2001} covers most of what we need. That is, if $G$ is a
 topological group (satisfying a mild topological hypothesis which is satisfied by all Lie groups) there is a functorially
 defined pointed CW complex $\B G$ and a principal $G$-bundle $\mathrm{E} G \to \B G$. If $X$ is any pointed CW complex, then there is a natural
 bijection of pointed sets between $[X, \B G]$ and $\Hoh^1(X, G)$, where the former denotes the set of homotopy classes of pointed maps between the
 spaces in question. The correspondence is given by assigning to a map $\xi : X \to \B G$ the pull-back of $\mathrm{E} G$ along $\xi$. It is a
 theorem that the isomorphism class of the resulting $G$-bundle depends on $\xi$ only up to homotopy. One may replace $\B G$ by a homotopy
 equivalent CW complex without harming this bijection. If $G$ is an abelian group, then $\B G$ is again a topological group, and we may form $\B^2G$, $\B^3G$
 and so on. If $G$ is discrete, then $\B^iG$ is homotopy equivalent to
 the Eilenberg--MacLane space $K(G,i)$, and $[X, K(A,i)]$ is
 isomorphic to the reduced cohomology group $\tilde \Hoh^i(X, A)$.

 \subsection{The {\boldmath{$2$}}-torsion of the topological Brauer group.}
 \label{subsection:torsion-of-brauer-grp}

The relationship between the topological Brauer group, $\Br_\topo(X)$, and the cohomology groups $\Hoh^2(X,
\CC^*)_\tors$ and $\Hoh^3(X, \ZZ)_\tors$ shall become important in Section \ref{section:back-to-algebra}, so we discuss
it here.

 First of all, there is an isomorphism $\Hoh^3(-, \ZZ)\iso \Hoh^2(-, \CC^\times)$. This arises from the exponential short exact sequence of
 topological groups $0 \to \ZZ \to \CC \to \CC^\times \to 1$, where $\CC$ is a contractible topological group, so
 $\Hoh^{\geq 1}(-, \CC)$ vanishes.
Often when
 $\Hoh^*(-, \CC)$ is written elsewhere, $\CC$ is to be understood as a discrete ring, so that $\Hoh^*(-, \CC) \iso
 \Hoh^*(-, \ZZ) \tensor_\ZZ \CC$; this is not the case here.


 Secondly, in general $\Br_\topo(-)$ is a subfunctor of $\Hoh^3(-,
 \ZZ)$, see \cite[Prp.~1.4]{grothendieck1968-a}. It is the subfunctor
 consisting of those classes that lie in the image of the coboundary
 maps $\Hoh^1(-, \PGL_n(\CC)) \to \Hoh^2(-, \CC^*)$ arising from the
 exact sequence $\CC^\times\to \GL_n(\CC)\to \PGL_n(\CC)$. Whenever
 $X$ has the homotopy type of a finite CW complex, $\Br_\topo(X)$ may
 be identified with the torsion subgroup of $\Hoh^3(X, \ZZ)$, this
 being a theorem of Serre \cite[Thm.~1.6]{grothendieck1968-a}. While
 there are spaces for which $\Br_\topo(X) \subsetneq
 \Hoh^3(X,\ZZ)_{\mathrm{tors}}$, for example $\B^2 \mu_2$ (see
 \cite[Cor.\ 5.10]{antieau2014}), in the sequel we shall deal only
 with classes $\alpha \in \Hoh^3(X, \ZZ)_\tors$ that arise as the
 pull-back $f^*(\alpha_0)$ of a canonical class in $\Hoh^2(\B \PGL_n,
 \CC^\times)_\tors$ corresponding to the universal $\PGL_n$--bundle on
 $\B\PGL_n$. Such a class $\alpha$ necessarily lies in $\Br_\topo(X)$.

\medskip

 In the sequel, we generally write $\mu_2$ for the group of square roots of $1$. We will use $\Hoh^*(-, \mu_2)$
 and $\Hoh^*(-, \ZZ/2)$ interchangeably.

 Consider the group $\mu_2 $ embedded as scalar matrices $\{\pm I_{2n}\}$ in $\GL_{2n}(\CC)$. The inclusion $\mu_2 \to \GL_{2n}(\CC)$
 factors through either of the two standard inclusions $\Og_{2n}(\CC)$ and $\Sp_{2n}(\CC)$, and it follows that there are natural maps
 $\Og_{2n}(\CC)/\mu_2 \to \GL_{2n}(\CC)/\mu_2 \to \PGL_{2n}(\CC)$ and $\Sp_{2n}(\CC)/\mu_2 \to \GL_{2n}(\CC)/\mu_2 \to
 \PGL_{2n}(\CC)$.

 We set about explaining how a principal $\Og_{2n}(\CC)/\mu_2$- or $\Sp_{2n}(\CC)/\mu_2$-bundle gives rise to a $2$-torsion Brauer
 class, cf.\ \cite[p.~214]{parimala_srinivas:brauer_group_involution}. Suppose $G \subset \GL_{2n}(\CC)$ is a subgroup such that $\mu_2 \subset G$. There is an associated map of short
 exact sequences
 \begin{equation}
 \label{eq:iv}
 \xymatrix{ 1 \ar[r] & \mu_2 \ar[d] \ar[r] & G \ar[d] \ar[r] & G/\mu_2 \ar[d] \ar[r] & 1 \\ 1 \ar[r] & \Gm \ar[r] &
 \GL_{2n} \ar[r] & \PGL_{2n} \ar[r] & 1 }
 \end{equation}
and therefore, for any CW complex $X$, a natural long exact sequences of cohomology groups
 yielding a commutative diagram:
 \begin{align}
 \label{eq:ii}
 &\xymatrix{ \Hoh^1(X, G/\mu_2) \ar[rr] \ar[d] & & \Hoh^2(X, \ZZ/2) \ar[d] \\ \Hoh^1(X, \PGL_{2n}(\CC)) \ar[r] & \Br_\topo(X)
 \ar@{^(->}[r] & \Hoh^2(X, \CC^*) \ar^{\iso}@{-}[r] & \Hoh^3(X, \ZZ).
 }
 \end{align}
That is to say, associated to an element in $\Hoh^1(X, G/\mu_2)$, there is an element of $\Hoh^3(X, \ZZ)$ which lies in
$\Br_{\topo}(X)$ and also in the image of a map from $\Hoh^2(X, \ZZ/2)$. It is therefore $2$-torsion.

There is a diagram of short exact sequences
\[ \xymatrix@C=30pt{ 0 \ar[r] & \ZZ \ar^{\times 2}[r] \ar@{=}[d] & \ZZ \ar[r] \ar^{\times i \pi}[d] & \ZZ/2 \ar^{x
    \mapsto (-1)^x}[d] \ar[r] & 0 \\ 0 \ar[r] & \ZZ \ar[r] & \CC \ar_{\exp}[r] & \CC^* \ar[r] & 1 } \] from which it
follows that the composite map $\Hoh^2(-, \ZZ/2 )
\to \Hoh^2(- , \CC^*) \to \Hoh^3(- , \ZZ)$ appearing in
Diagram \ref{eq:iv} is simply the unreduced Bockstein homomorphism associated to $\ZZ \xrightarrow{\times 2} \ZZ \to
\ZZ/2$, and denoted
\[\beta_2:\Hoh^2(-, \ZZ/2 )\to \Hoh^3(- , \ZZ)\ .\]

 \begin{remark} \label{rem:iv}
 For technical reasons, we find it easier to work with the simply connected group $\SL_{2n}(\CC)$ rather than $\GL_{2n}(\CC)$. For
 comparison, therefore, we should work with $\SOg_{2n}(\CC)$ rather than $\Og_{2n}(\CC)$.

 The inclusion of scalar matrices $\mu_2 \to \Og_{2n}$ factors as $\mu_2 \to \SOg_{2n} \to \Og_{2n}$ and there is
 a diagram
 \[ \xymatrixcolsep{4.5pc}
 \xymatrix{ \mu_2 \ar@{=}[r] \ar[d] & \mu_2 \ar[r] \ar[d] & 1 \ar[d] \\ 
 \SOg_{2n}(\CC) \ar[d] \ar[r] & \Og_{2n}(\CC) \ar^{\det}[r] \ar[d] & \mu_2 \ar[d] \\ 
 \SOg_{2n}(\CC) / \mu_2 \ar[r] & \Og_{2n}(\CC)/\mu_2 \ar^{\det}[r] & \mu_2 }\]
 in which every row and column is exact. In particular, the diagram
 \[ \xymatrix@C=7pt{\Hoh^1(Y, \SOg_{2n}(\CC)/\mu_2) \ar[dr] \ar[rr] & & \Hoh^1(Y, \Og_{2n}(\CC)/\mu_2 )\ar[dl]\ar[rrr] & & & \Hoh^1(Y, \ZZ/2) \\ &
 \Hoh^2(Y, \ZZ/2) } \] commutes and the top row is an exact sequence of pointed sets. If $\Hoh^1(Y, \ZZ/2) = 0$, then the
 problem of finding a lift of an element of $\Hoh^2(Y, \ZZ/2)$ to a class in $\Hoh^1(Y, \Og_{2n}(\CC)/\mu_2)$ is equivalent to the
 problem of lifting to $\Hoh^1(Y, \SOg_{2n}(\CC)/\mu_2)$.

 Since $\Sp_{2n}(\CC) \subset \SL_{2n}(\CC)$, no modification of the group is required in the symplectic case.
 \end{remark}

 Restricting our attention to the special-linear setting, we have groups: $\SL_{2n}(\CC)/\mu_2$, $\Sp_{2n}(\CC)/\mu_2$
 and $\SOg_{2n}(\CC)/\mu_2$.
 There are induced diagrams
 \begin{equation} \label{eq:i} \xymatrix@C=7pt{ \B \Sp_{2n}(\CC)/\mu_2 \ar[rr] \ar_{\xi_s}[dr] & & \B \SL_{2n}(\CC)/ \mu_2 \ar^{\xi}[dl]
 \\ & \B^2 \mu_2 }, \quad \xymatrix@C=7pt{ \B \SOg_{2n}(\CC)/\mu_2 \ar[rr] \ar_{\xi_o}[dr] & & \B \SL_{2n}(\CC)/ \mu_2 \ar^{\xi}[dl]
 \\ & \B^2 \mu_2 .}
 \end{equation}
 These correspond to diagrams in cohomology for a CW complex $X$:
 \begin{equation}
 \label{eq:iii}
 \xymatrix@C=16pt{ \Hoh^1(X, \Sp_{2n}(\CC)/\mu_2) \ar[r] \ar[d] & \Hoh^2(X, \mu_2) \ar@{=}[d] \\ \Hoh^1(X, \SL_{2n}/
 \mu_2) \ar[r] & \Hoh^2(X, \mu_2) }  
 \quad 
 \xymatrix@C=16pt{ \Hoh^1(X, \SOg_{2n}(\CC)/\mu_2) \ar[r] \ar[d] & \Hoh^2(X, \mu_2) \ar@{=}[d] \\ \Hoh^1(X, \SL_{2n}/
 \mu_2) \ar[r] & \Hoh^2(X, \mu_2) }
 \end{equation}

We now summarize the topological argument.   In the next section, we will construct a pointed CW complex $Y$ such that $\Hoh^1(Y, \ZZ/2) = 0$ and $\Hoh^2(Y,\ZZ)=0$, 
 and a $2$-torsion class $\alpha \in \Br_\topo(Y) \subset
 \Hoh^3(Y, \ZZ)$.

 The class $\alpha$ admits a lift to a class \[\zeta \in \Hoh^2(Y, \ZZ/2)\] such that the unreduced Bockstein $\beta_2(\zeta)$ is $\alpha$. The
 condition $\Hoh^2(Y, \ZZ) = 0$ ensures the class $\zeta$ is unique. It will be the case that $\zeta$ is in the image of a map $\Hoh^1(Y ,
 \SL_4(\CC) / \mu_2) \to \Hoh^2(Y, \ZZ/2)$, so that in particular, $\alpha$ is in the image of a map $\Hoh^1(Y, \PGL_4(\CC)) \to \Hoh^2(Y,
 \CC^*) = \Hoh^3(Y, \ZZ)$; cf.\ Diagram \eqref{eq:ii}. That is, the index of $\alpha$ divides $4$.

 The class $\zeta$ is represented by a map $\zeta : Y \to \B^2 \mu_2$. We will show that there is no factorization in homotopy of $\zeta$ as a
 map $Y \to \B G\to \B^2 \mu_2$, where $G$ is either $\SOg_{2n}(\CC)/\mu_2$ or $\Sp_{2n}(\CC)/\mu_2$ for $n \not \equiv 0 \pmod 4$, thus
 showing that $\zeta \in \Hoh^2(Y, \ZZ/2)$ is not in the image of the natural map $\Hoh^1(Y, G) \to \Hoh^2(Y, \ZZ/2)$. It follows
 that $\alpha$ is not in the image of the natural map $\Hoh^1(Y, G) \to \Hoh^3(Y, \ZZ)$ for such $G$, and this shows that $\alpha$
 is not the Brauer class of an Azumaya algebra with involution of degree $2n$ where $n \not \equiv 0 \pmod 4$.

\subsection{The Cohomology of $\B \SL_{2n}/{\mu}_2$}

In order to simplify the notation, we have written in this section
$\SL_{2n}$ in place of $\SL_{2n}(\CC)$ and
similarly for $\GL_{2n}$.

The Cartan--Iwasawa--Malcev Theorem, see for instance
\cite{iwasawa1949}, says that the Lie groups $\SL_{2n}(\CC)$,
$\Og_{2n}(\CC)$, and $\Sp_{2n}(\CC)$ are homotopy equivalent to maximal
compact subgroups, which may be taken to be the compact Lie groups
$\SU_{2n}$, $\Og_{2n}(\RR)$, and $\Sp_{2n}$ respectively.

Our main technical tool is the low-degree part of the Serre spectral sequence in integral cohomology associated to the fiber
sequence $\B\SL_{2n} \to \B\SL_{2n} / \mu_2 \to \B^2 \mu_2$. A portion of this is presented in Figure~\ref{fig:1}.

Necessary background for the Serre spectral sequence \textit{per se} may be found in \cite[Chp.\
5,6]{mccleary2001}. The $\Eoh_2$ page of the Serre spectral sequence of a fiber sequence $F \to E \to B$, where $B$ is
simply connected, takes the form $\Eoh_2^{p,q} = \Hoh^p(B, \Hoh^q(F, \ZZ))$, the differentials satisfy $d_r:
\Eoh_r^{p,q} \to \Eoh_r^{p+r, q-r+1}$, and it converges strongly to $\Hoh^{p+q}(E, \ZZ)$. In the cases we encounter,
$\Eoh_2^{p,q} = \Hoh^p(B, \ZZ) \tensor_\ZZ \Hoh^q(F, \ZZ)$, since $\Hoh^*(F, \ZZ)$ will be a free abelian group.

The cohomology of $\B\SL_{2n}$ is well-known, and may be deduced from \cite[Ex.\ 5.F,~Thm.\ 6.38]{mccleary2001},
along with the observation that $\B \SU_{2n}$ is homotopy equivalent to $\B \SL_{2n}$. The homology  of $\B^2
\mu_2$ can be found in \cite{cartan1954}, and the cohomology may be deduced from there using the Universal Coefficients
Theorem.
\begin{figure}[h]
 \centering
\[ \xymatrix@R=5pt@C=10pt{ c_2\ZZ = \Hoh^4(\B \SL_{2n}, \ZZ) \ar^{d_5}[ddddrrrrr] \\ 0 \\ 0 \\ 0 \\ \ZZ & 0 & 0 & q_1\ZZ/2 & 0 & q_2\ZZ/4. }
\]
 \caption{A portion of the Serre spectral sequence in integral cohomology associated to $\B \SL_{2n} \to \B \SL_{2n} / \mu_2 \to \B^2\mu_2$.}
 \label{fig:1}
\end{figure}

\begin{lemma} \label{lem:H2lem}
 The map $\xi: \B \SL_{2n}/ \mu_2 \to \B^2 \mu_2$, appearing in Diagram (\ref{eq:i}), represents a generator of
 \[ \Hoh^2(\B \SL_{2n}/ \mu_2 , \ZZ) \iso \ZZ/2.\]
\end{lemma}
\begin{proof}
 By considering Figure \ref{fig:1}, we see that $\xi$ induces an isomorphism on $\Hoh^{\le 3}(-, \ZZ)$. It follows
 that it induces an isomorphism on $\Hoh^2(-, \ZZ/2)$, but $\Hoh^2(\B^2 \mu_2, \ZZ/2) = \iota\ZZ/2$, where $\iota$
 is represented by the identity map $\B^2 \mu_2 \to \B^2 \mu_2$. The result follows.
\end{proof}

We concern ourselves with the calculation of the $d_5$ differential by comparison with the case of the (compact) maximal torus of
diagonal unitary matrices
$ST_{2n} \subset \SL_{2n}$. Write $T_{2n}$ for the (compact) maximal torus of diagonal unitary matrices in $\GL_{2n}$. The following
descriptions are well known:
\begin{align*}
 \Hoh^*(\B T_{2n} , \ZZ) & \iso \ZZ[ \theta_1, \dots, \theta_{2n}], \quad |\theta_i| = 2, \\
 \Hoh^*(\B \GL_{2n}, \ZZ) & \iso \ZZ[ c_1, \dots, c_{2n}] \subset \ZZ[ \theta_1, \dots, \theta_{2n}], \quad c_i =
 \sigma_i(\theta_1, \dots, \theta_{2n}), \\
 \Hoh^*(\B ST_{2n}, \ZZ) & \iso \ZZ[\theta_1, \dots, \theta_{2n}]/(\theta_1 + \theta_2 + \dots + \theta_{2n}), \\
 \Hoh^*(\B \SL_{2n}, \ZZ) & \iso \ZZ[c_2, \dots, c_{2n}].
\end{align*}
A proof of the relation between $\Hoh^*(\B T_{2n} , \ZZ)$ and $\Hoh^*(\B \GL_{2n}, \ZZ)$ appears as
\cite[Thm.\ 18.3.2]{husemoller1994}, provided we recall that the noncompact groups appearing here are weakly equivalent to the
compact groups appearing there. The reduction to the special linear case is an easy Serre spectral sequence argument.

The cohomology $ \Hoh^*(\B ST_{2n}, \ZZ)$ is the polynomial algebra on $2n-1$ generators, which we may as well take to be
the images of $\theta_1, \dots, \theta_{2n-1}$, all lying in degree $2$. There
is a reduced augmentation map 
\[ \phi:\Hoh^2(\B ST_{2n}, \ZZ) \to \ZZ/2\] given by $\theta_i \mapsto 1$ for all $i$.

There is a comparison of fiber sequences
\begin{equation} \label{eq:compSeq}
\xymatrix{ \B ST_{2n} \ar[r] \ar[d] & \B ST_{2n} / \mu _2 \ar[d] \ar[r] & \B^2 \mu_2 \ar[d] \\ \B \SL_{2n} \ar[r] & \B \SL_{2n}/ \mu_2 \ar[r] & \B^2 \mu_2. }
\end{equation}

\begin{lemma}
 The map $\B ST_{2n} \to \B ST_{2n}/\mu_2$ identifies $\Hoh^*(\B ST_{2n}/
\mu_2, \ZZ)$ with the subring of $\Hoh^*(\B ST_{2n}, \ZZ)$ generated by 
$\ker (\phi:\Hoh^2(\B ST_{2n}, \ZZ/2) \to \ZZ/2)$.
\end{lemma}
\begin{proof}
 Since $ST_{2n} / \mu_2$ is a compact connected abelian Lie group of dimension $2n-1$, it is again a torus. Therefore,
 both $\Hoh^*(\B ST_{2n}, \ZZ)$ and $\Hoh^*(\B ST_{2n} / \mu_2, \ZZ)$ are polynomial rings over $\ZZ$ on $2n-1$
 generators in dimension $2$. The assertion of the lemma reduces to the claim that $\Hoh^2(\B ST_{2n}/\mu_2 , \ZZ) \to
 \Hoh^2(\B ST_{2n}, \ZZ)$ is exactly $\ker \phi$.

 Consider the inclusion maps $\mu_2 \overset{\iota}{\to} ST_{2n} \to T_{2n}$. These are compatible with the $\Sigma_{2n}$ action
 on $T_{2n}$, given by permuting factors in $T_{2n}$. The induced action of $\Sigma_{2n}$ is transitive on the classes
 $\theta_i \in \Hoh^2(\B T_{2n}, \ZZ)$, and therefore also on their reductions to $\Hoh^2(\B ST_{2n}, \ZZ)$, while
 $\Sigma_{2n}$ acts trivially on $\Hoh^2(\B \mu_2, \ZZ)$. We deduce that $\iota^*( \theta_i ) = \iota^*(\theta_j)$
 where $1 \le i,j \le 2n-1$.

 \begin{figure}[h]
 \centering
 \[ \xymatrix@R=5pt@C=10pt{ \Hoh^2(\B \mu_2, \ZZ) \iso \ZZ/2 \\ 0 \\ \ZZ & 0 & \Hoh^2(\B ST_{2n}/\mu_2, \ZZ) & 0 } \]
 \caption{A portion of the Serre spectral sequence in integral cohomology associated to $\B \mu_2 \to \B ST_{2n} \to \B ST_{2n}/\mu_2$.}
 \label{fig:2augLemma}
 \end{figure}
 Now we examine the Serre spectral sequence associated to the fiber sequence $\B \mu_2 \to \B ST_{2n} \to \B ST_{2n}
 /\mu_2$, as given in Figure \ref{fig:2augLemma}. It collapses in the illustrated range, and we are left with an
 extension problem
 \[ 0 \to \Hoh^2(\B ST_{2n}/ \mu_2, \ZZ) \to \Hoh^2(\B ST_{2n}, \ZZ) \overset{\iota^*}{\to} \Hoh^2( \B \mu_2, \ZZ)
 \to 0. \]

 Since $\iota^* \neq 0$, it follows that $\iota^*(\theta_i) = 1$ for at least one, and therefore for all, values of $i
 \in \{0, \dots, 2n-1\}$. Therefore, $\iota^* = \phi$, and the result follows.
\end{proof}

Explicitly, one may present $\Hoh^*(\B ST_{2n}/ \mu_2, \ZZ)$ as the subring of $\Hoh^*(\B ST_{2n}, \ZZ)$ generated by
the reductions modulo the ideal $(\theta_1 + \theta_2 + \dots + \theta_{2n})$ of $\{2 \theta_1\} \cup \{ \theta_i -
\theta_1 \}_{i=2}^{2n}$.

For convenience, we write $\Hoh^*(\B ST_{2n}, \ZZ) = \ZZ[ x_1, y_2, \dots, y_{2n-1}]$, where $x_1$ is the reduction
modulo the ideal $(\theta_1 + \theta_2 + \dots + \theta_{2n})$
of $\theta_1$, and $y_i$ is the reduction modulo the ideal $(\theta_1 + \theta_2 + \dots + \theta_{2n})$ of $\theta_i - \theta_1$. Observe that the reduction of $\theta_{2n}$ modulo
$(\theta_1 + \theta_2 + \dots + \theta_{2n})$ is $-(2n-1)x_1-\sum_{i=2}^{2n-1} y_i$. In this notation, $\Hoh^*(\B ST_{2n}/ \mu_2, \ZZ) = \ZZ[2x_1, y_2, \dots, y_{2n-1}]$.

\begin{figure}[h]
 \centering
\[ \xymatrix@R=5pt@C=10pt{ \displaystyle x_1^2 \ZZ \oplus \bigoplus_{i=1}^{2n-1} x_1y_i \ZZ \oplus \bigoplus_{i \le j } y_i y_j \ZZ
 \ar@/^4em/^{d_5}[ddddrrrrr] \ar_{d_3}[ddrrr]
 \\ 0 \\ \displaystyle x_1 \ZZ \oplus \bigoplus_{i=2}^{2n-1} y_i \ZZ \ar_{d_3}[ddrrr] & & & (\ZZ/2)^{2n-1} \\ 0 \\ \ZZ & 0 & 0 & q_1\ZZ/2 & 0 & q_2\ZZ/4 }
\]
 \caption{A portion of the Serre spectral sequence in integral cohomology associated to $\B ST_{2n} \to \B ST_{2n} / \mu_2 \to \B^2\mu_2$.}
 \label{fig:2}
\end{figure}

Since the spectral sequence of Figure~\ref{fig:2} converges to $\ZZ[2x_1, y_2, \dots, y_{2n-1}]$, we deduce that
$x_1^2 \ZZ \cap
\mathrm{E}_\infty^{0,4} = 4x_1^2 \ZZ$, and since $d_3(x_1^2) = 2 x_1 d_3(x_1) = 0$, it follows that $d_5(x_1^2) = q_2$ up to a negligeable
choice of sign.

We turn to the calculation of the $d_5$ differential in the spectral sequence of Figure \eqref{fig:1}.

\begin{lemma} \label{lem:d5c2}
 In the spectral sequence of Figure~\ref{fig:1}, $d_5(c_2) = \gcd(n, 4) q_2$, up to a negligeable choice of sign.
\end{lemma}
\begin{proof}
 There is a comparison map from the spectral sequence of Figure \ref{fig:1} to the spectral sequence of Figure~\ref{fig:2}.

 One has the following expansion in $\Hoh^*(\B ST_{2n}, \ZZ)$
\begin{align*}
 c_2 &= \sigma_2(\theta_1, \dots, \theta_{2n}) \\
 & = \sigma_2\Big( x_1, y_2 + x_1, \dots, y_{2n-1} + x_1, -(2n-1) x_1 - \sum_{i=2}^{2n-1} y_i \Big) \\
 & = \binom{2n-1}{2} x_1^2 - (2n-1)(2n-1) x_1^2 -2nx_1(y_2+ \dots + y_{2n-1}) + p(y_2, \dots, y_{2n-1}) \\
 & = -(2n-1)nx_1^2 -2nx_1(y_2+ \dots + y_{2n-1}) + p(y_2, \dots, y_{2n-1})
\end{align*}
where $p$ is a homogeneous polynomial of degree $2$.

Since $d_5(y_i) = 0$ and $d_5(2x_1) = 0$ in the sequence of Figure~\ref{fig:2}, it follows that $d_5(c_2) = -(2n-1)n
d_5(x_1^2) = u\gcd(n,4) q_2$, for some unit $u$, in the sequence of Figure~\ref{fig:2}, and by comparison also in that of Figure~\ref{fig:1}.
\end{proof}

\subsection{Obstructions to self maps of approximations to $\B \SL_{2n}/\mu_2$}

We remind the reader that a based space $(X, x_0)$ is $n$-\emph{connected} if it is connected and the homotopy groups $\pi_i(X, x_0)$ are
trivial for $i \le n$. In the sequel we drop the basepoints from the notation and all spaces will be assumed based. A
map of spaces $f: X \to Y$ is an $n$-equivalence if the homotopy fiber, $\hofib f$, is $(n-1)$-connected. In practice,
this means that $f_*: \pi_i(X) \to \pi_i(Y)$ is a bijection for $i < n$ and is surjective for $i=n$. If $f: X \to Y$ is
an $n$-equivalence, then $f^*: \Hoh^i(Y, R) \to \Hoh^i(X,R)$ is an isomorphism when $i< n$; this follows from 
\cite{whitehead2012}*{Thm.~IV.7.13} and the Universal Coefficient Theorem.

\begin{proposition} \label{pr:techProp}
Suppose $Y_1, Y_2$ are two spaces each equipped with $7$-equivalences $Y_1 \to \B\SL_{2n_1}/ \mu_2$, $Y_2 \to \B\SL_{2n_2}/ \mu_2$,
where the $n_i$ are not multiples of $4$, and $n_1 \ge 2$ is even. Any map $f: Y_1 \to Y_2$ inducing an isomorphism $f^*:\Hoh^2(Y_2 , \ZZ/2) \to \Hoh^2(Y_1,
\ZZ/2) \iso \ZZ/2$ induces an isomorphism $f^*: \Hoh^6(Y_2, \QQ) \to \Hoh^6(Y_1, \QQ)$.
\end{proposition}
\begin{proof}
 Let $Y_i'$ denote the homotopy fibers of the composite maps $Y_i \to \B\SL_{2n_i}/ \mu_2 \to \B^2 \mu_2$. These are $7$-equivalent
 to $\B\SL_{2n_i}$ by comparison with the case of $\B\SL_{2n_i} \to \B\SL_{2n_i}/\mu_2 \to \B^2 \mu_2$. The low-degree parts of the Serre spectral
 sequences in integral cohomology associated to $Y_i' \to Y_i \to \B^2 \mu_2$ necessarily take the form shown in Figure \ref{fig:3}.
\begin{figure}[h]
 \centering
\[ \xymatrix@R=5pt@C=10pt{ \ZZ = \Hoh^4(Y_i', \ZZ) \ar^{d_5}[ddddrrrrr] \\ 0 \\ 0 \\ 0 \\ \ZZ & 0 & 0 & q_1\ZZ/2 & 0 & q_2\ZZ/4 }
\]
 \caption{A portion of the Serre spectral sequence in integral cohomology associated to $Y_i' \to Y_i \to \B^2\mu_2$.}
 \label{fig:3}
\end{figure}
This coincides with the Serre spectral sequence associated to $\B\SL_{2n_i} \to \B\SL_{2n_i}/\mu_2 \to \B^2 \mu_2$,
which is presented in Figure \ref{fig:1}.

By comparison with the case of $\B\SL_{2n_i} \to \B\SL_{2n_i}/\mu_2 \to \B^2 \mu_2$ and use of Lemma \ref{lem:d5c2}, the differential $d_5$ takes a generator to $2 q_2$ when $i=1$.

The fact that $f$ induces an isomorphism on $\Hoh^2(Y_i, \ZZ/2)$
for $i\in \{1,2\}$ implies that there is a homotopy-commutative diagram
\begin{equation}\label{EQ:fibered-seq-diagram}
 \xymatrix{ Y_1' \ar^{f'}[d] \ar[r] & Y_1 \ar^f[d] \ar[r] & \B^2\mu_2 \ar[d] \\
Y_2 ' \ar[r] & Y_2 \ar[r] & \B^2\mu_2 } 
\end{equation}
and an attendant map of spectral sequences which induces the identity on $\mathrm{E}_2^{\ast, 0}$. It follows by
considering the $d_5$ differential in Figure \ref{fig:3} that $f'^* : \Hoh^4( Y_2'
, \ZZ) \to \Hoh^4(Y_1', \ZZ)$ is multiplication by an odd integer, and so $f'^*: \Hoh^4(Y_2', \ZZ/2) \to \Hoh^4(Y_1',
\ZZ/2)$ is an isomorphism.

We have $\Hoh^4(\B\SL_{2n_i}, \ZZ/2) = \bar c_2 \ZZ/2 $, where $\bar x$ denotes the reduction modulo $2$ of the integral
cohomology class $x$. We now make use of the action of the Steenrod algebra on $\ZZ/2$ cohomology. The properties of
this action may be found in \cite[\S.\ 4L]{hatcher}. We compute the action of the Steenrod algebra on $\Hoh^*(\B
\SL_{2n}, \ZZ/2)$ by comparison with $\Hoh^*(\B T_{2n}, \ZZ/2)$. In $\Hoh^*(\B T_{2n}, \ZZ/2)$, one has $\Sq^2 \bar
\theta_i = \bar\theta_i^2$ for all $i$, and using the axioms, we calculate that \[ \Sq^2 \sigma_2(\bar \theta_1, \dots, \bar
\theta_{2n_i}) = \sigma_1(\bar \theta_1, \dots, \bar \theta_{2n_i}) \sigma_2(\bar \theta_1, \dots, \bar \theta_{2n_i}) +
\sigma_3(\bar \theta_1, \dots, \bar \theta_{2n_i}).\] Therefore, upon reducing to the case of $\B \SL_{2n_i}$, we obtain
$\mathrm{Sq}^2 \bar c_2 = \bar c_3$, the reduction of the third Chern class.

By naturality of the Steenrod operations, it follows that
$f'^*: \Hoh^6(Y_2' , \ZZ/2) \to \Hoh^6(Y_1', \ZZ/2)\iso \ZZ/2$ is an isomorphism. Since $\Hoh^6(Y_i', \ZZ/2)$ is, in
each case, the reduction modulo $2$ of the free abelian group $\Hoh^6(Y_i, \ZZ) \iso \ZZ$, it follows that the natural maps $\Hoh^6(Y_i , \QQ ) \to \Hoh^6(Y_i', \QQ)$ are isomorphisms

Using the Serre spectral sequence, for instance, one deduces that $\Hoh^*(\B \Gamma, \QQ) \iso \Hoh^*(\B \Gamma/\mu_2,
\QQ)$ where $\Gamma$ is a topological group containing $\mu_2$ as a subgroup, so $f'^*: \Hoh^6(Y_2', \QQ) \to
\Hoh^6(Y_1', \QQ)$ is also an isomorphism. That $f^*:\Hoh^6(Y_2,\QQ)\to \Hoh^6(Y_1,\QQ)$ is an isomorphism
now follows 
from the previous paragraph and left square of diagram \eqref{EQ:fibered-seq-diagram}.
\end{proof}

\begin{proposition} \label{something}
 Suppose $Y \to \B\SL_{4}/\mu_2$ is a $7$-equivalence. Let $\eta : Y \to \B^2\mu_2$ denote the composite $\xi \circ (Y \to
 \B\SL_{2n}/\mu_2)$. Let $G$ be one of the groups $\Sp_{2n}/\mu_2$ or $\SOg_{2n}/\mu_2$, where $n$ is not divisible by $4$. There is no map
 $f:Y \to \B G$ making the following diagram commute in homotopy
 \[ \xymatrix{ Y \ar^f[rr] \ar_\eta[dr] & & \B G \ar^{\xi_{?}}[dl]\\ & \B^2 \mu_2 }. \]
\end{proposition}

\begin{proof}
 The cohomology $\Hoh^*(\B \Sp_{2n}, \ZZ)$ is a polynomial ring on classes in dimensions $\{4,8,\dots, 4n\}$; this
 calculation is classical and is to be found in \cite{borel1953} following Proposition 29.2. The cohomology ring $\Hoh^*(\B
 \SOg_{2n}, \ZZ)$ is calculated in \cite{brown1982}. In each case, $\Hoh^6( \B G, \ZZ)$ is torsion, so it follows in
 our situation that $\Hoh^6(\B G, \QQ) = 0$.

 Now we turn to disproving the existence of a map $f$. If there were such a map, then by composing $Y \to \B G$ with the
 appropriate map in Diagram \eqref{eq:i}, we could construct
 a map $g: Y \to\B G \to \B \SL_{4}/\mu_2$ over $\B^2 \mu_2$. 
 By Lemma~\ref{lem:H2lem}, we know that both $\eta: Y\to \B^2\mu_2$ and $\B\SL_4/\mu_2\to \B^2\mu_2$ represent a generator of
 $\Hoh^2(-, \ZZ/2)$, and so $g^*:\Hoh^2(\B\SL_{4}/\mu_2,\ZZ/2)\to \Hoh^2(Y,\ZZ/2)$ is an isomorphism.
 By Proposition \ref{pr:techProp}, the map $g$ induces an isomorphism $\Hoh^6( \B
 \SL_{2n}/\mu_2, \QQ) \to \Hoh^6( Y, \QQ) \iso \QQ$, factoring through $\Hoh^6(\B G, \QQ) \iso 0$, a contradiction.
\end{proof}

\section{Back to Algebra}
\label{section:back-to-algebra}

In this section, the notation reverts to $\SL_{2n}, \GL_{2n}$ for a
group scheme and $\SL_{2n}(\CC), \GL_{2n}(\CC)$ for the Lie group of
complex points.  We can now prove the main Theorem~\ref{TH:main},
which we rephrase.

\begin{theorem}\label{TH:no-involution}
There exists a nonsingular affine variety $\Spec R$ over $\CC$ with an
Azumaya algebra $A$ of period $2$ and degree $4$ such that if $B$ is
any Azumaya algebra with involution and $[B] = [A]$ then $8
| \deg (B)$.
\end{theorem}

The construction is similar to that of \cite[Thm.\ 1.1]{antieau2014-a}.
\begin{proof}
 Let $V$ be an algebraic linear representation of $\SL_4/ \mu_2$ over $\CC$ such that $\SL_4/\mu_2$ acts freely outside an invariant closed
 subscheme $S$ of codimension at least $5$, and such that $(V-S)/(\SL_4/\mu_2)$ exists as a smooth quasi-projective complex variety. Such
 representations exist by~\cite{totaro1999}*{Rem.\ 1.4}. As the codimension of $S$ is at least $5$, the space $(V-S)(\CC)$ is
 $2(5)-2 = 8$--connected.

Since we would like to have an affine example in particular, we use the Jouanolou device, \cite{jouanolou1973}, to replace $(V-S)/(\SL_4/\mu_2)$ by an affine
vector bundle torsor $p:\Spec R \to (V-S)/(\SL_4/\mu_2)$. In order to simplify the notation, we write $Y = \Spec R$. The map $Y(\CC) \to
(V-S)/(\SL_4/\mu_2)(\CC)$ is a homotopy equivalence.

We pull the evident $(\SL_4/ \mu_2)$-torsor on $(V-S)/(\SL_4/\mu_2)$ back along
$p$, giving an $\SL_4/\mu_2$-torsor, $T$, on $Y$.
There is a map $Y(\CC) \to \B(\SL_4(\CC)/\mu_2)$, classifying the complex realization $T(\CC)$, and the map $Y(\CC) \to
 \B(\SL_4(\CC)/\mu_2)$ is an $8$--equivalence of topological spaces.

 As explained in Section~\ref{subsection:torsion-of-brauer-grp}, the algebraic $(\SL_4/ \mu_2)$-torsor $T$ on $Y$ induces an  Azumaya algebra  $A$ over $Y$ of degree $4$ whose image in $\Br(Y)$ is $2$-torsion.
 Consider the topological Azumaya algebra $A(\CC)$ over $Y(\CC)$. Its image 
 in $\Hoh^2(Y(\CC), \ZZ/2)$ is classified by a map $\eta:Y(\CC)\to \B^2 \mu_2$,
 which factors through the $8$-equivalence $Y(\CC)\to \B \SL_4(\CC)/\mu_2$.
 Therefore, by Lemma~\ref{lem:H2lem}, the map $\eta$ represents a generator of 
 $\Hoh^2(Y(\CC),\ZZ/2)\cong \ZZ/2$. Since $\Hoh^2(Y(\CC),\ZZ)=0$,
 the map $\beta_2:\Hoh^2(Y(\CC),\ZZ/2)\to \Hoh^3(Y(\CC),\ZZ)_\tors=\Br(Y(\CC))$ is injective,
 and
 therefore $A(\CC)$ is not split.
 

 Suppose $B$ is some Azumaya algebra on $Y$ equivalent to $A$, and $B$ is equipped with an involution of the first
 kind. Since $A$ is not split, the degree of $B$ is an even integer, $2n$. Then the topological realization $B(\CC)$ is equivalent to
 $A(\CC)$, and $B(\CC)$ is classified by a map $\beta: Y \to \B G$ where $G = \SOg_{2n}(\CC) / \mu_2$ or $\Sp_{2n}(\CC)/\mu_2$, depending
 on whether the involution is orthogonal or symplectic; here we use Remark \ref{rem:iv} to replace $\Og_{2n}(\CC)$ by $\SOg_{2n}(\CC)$. Since the Brauer classes of $A(\CC)$ and $B(\CC)$ are the same, there is a
 homotopy commutative diagram
 \[ \xymatrix{ Y(\CC) \ar^\beta[rr] \ar^{\eta}[dr] && \B G \ar^{\xi_?}[dl] \\ & \B^2 \mu_2} \]
 from which Proposition \ref{something} implies that $8 | 2n$.
\end{proof}
 We remark that once $\Spec R$ has been found of some, possibly large, dimension, the affine Lefschetz theorem (see~\cite{goresky-macpherson}*{Introduction,
\S2.2}), ensures we can replace it by a $7$-dimensional smooth affine variety.

 We further note that any algebra $A$ as in Theorem~\ref{TH:no-involution} has index $4$ and no zero divisors.
 Indeed, if $A'\in[A]$, then by the proof of Saltman's result given by
 Knus, Parimala, and Srinivas~\cite[\S4]{KnParSri90}, there
 is $A''\in[A]$ admitting an involution with $\deg A''=2\deg A'$. By assumption, $8\mid \deg A''$,
 and hence $4\mid\deg A'$. It follows that $\ind A=\mathrm{gcd}\{\deg A'\,:\, A'\in[A]\}$ is at least $4$.
 Next, let $K$ be the fraction field of $R$. Since $\Spec R$ is nonsingular,
 $\ind A\otimes_R K=\ind A$
 (\cite[Prp.~6.1]{antieau2014}). This implies $\deg A\otimes_R K=4=\ind A\otimes_RK$, and therefore $A\otimes_RK$ is a division algebra.

\section{Split Azumaya Algebras With Only Symplectic Involutions}

Let $X$ be a scheme, let $\calP$ be a locally free $\calO_X$-module of
finite rank, and let $\calL$ be a line bundle over $X$.  Given a
symmetric bilinear form $b : \calP \times \calP \to \calL$, there is
an associated even Clifford algebra $C_0(b)$ and a Clifford bimodule
$C_1(b)$ over the even Clifford algebra constructed by Bichsel
\cite{bichsel:thesis}, \cite{bichsel_knus:values_line_bundles}.  For a
summary of its main properties, see \cite[\S1.8]{auel:clifford},
\cite[\S1.2]{auel:surjectivity}, or \cite[\S1.5]{ABB:quadrics}.  In
particular, if $\calP$ has rank $n$, then $C_0(b)$ and $C_1(b)$ are
locally free $\calO_X$-modules of rank $2^{n-1}$.  There is a natural
vector bundle embedding $\calP \to C_1(b)$.  Assume that $b$ is
regular, i.e., the associated map $\calP \to \HHom(\calP,\calL)$ is an
isomorphism.  If moreover $n$ is even, then the center $Z(b)$ of
$C_0(b)$ is an \'etale quadratic $\calO_X$-algebra, and the left
action of $Z(b)$ on $C_1(b)$ is a twist of the right action by the
unique nontrivial $\calO_X$-algebra automorphism of $Z(b)$.  The class
of the associated \'etale quadratic cover $Z \to X$ defines the
discriminant class $d(b) \in \Hoh^1_\et(X,\ZZ/2\ZZ)$.

\begin{proposition}
\label{PR:binary-decomposable}
Let $b : \calP \times \calP \to \calL$ be a regular symmetric bilinear
form of rank 2 over a scheme $X$. If $b$ has trivial discriminant
then $\calP$ is a direct sum of line bundles.
\end{proposition}
\begin{proof}
This is a consequence of \cite[Cor.~5.6]{auel:clifford}, but we will
give a direct proof here for completeness. Since $\calP$ has rank 2,
the vector bundle embedding $\calP \to C_1(b)$ is an $\calO_X$-module
isomorphism, and similarly, the embedding $Z(b) \to C_0(b)$ is an
$\calO_X$-algebra isomorphism. As $b$ is regular, another property is
that $C_1(b)$ is an invertible right $C_0(b)$-module, i.e., $\calP$
has a canonical structure of an invertible $Z(b)$-module.
If $b$ has trivial discriminant, then $Z(b)$ is the split \'etale
quadratic algebra $\calO_X \times \calO_X$, and therefore there is a
decomposition of $\calP$ into a direct sum of two invertible
$\calO_X$-modules.
\end{proof}

This gives a nontrivial necessary condition for the existence of a
regular symmetric bilinear form on a rank 2 vector bundle $\calP$. On
the other hand, any rank 2 vector bundle $\calP$ has a canonical
regular skew-symmetric form $\calP \times \calP \to
\exterior^2 \calP$ defined by wedging.

\begin{corollary}
 \label{COR:no-orthogonal-involution}
 Let $X$ be a scheme such that $\Hoh_\etrm^1(X,\ZZ/2\ZZ)=0$. If
$A=\EEnd(\calP)$ is a split Azumaya algebra of degree 2 with
orthogonal involution over $X$, then $\calP$ is a direct sum
of line bundles. In particular, if $\calP$ is an indecomposable
vector bundle of rank 2 on $X$, then $\EEnd(\calP)$ carries a
symplectic involution but no orthogonal involution.
\end{corollary}
\begin{proof}
By a result of Saltman \cite[Thm.~4.2a]{Sa78}, any orthogonal (resp.\
symplectic) involution on $\EEnd(\calP)$ is adjoint to a regular
symmetric (resp.\ skew-symmetric) bilinear form $b : \calP \times
\calP \to \calL$ in the sense of Example~\ref{EX:involutions}.
If $A$ admits an orthogonal involution and $b$ is the corresponding form,
then the discriminant $d(b)\in\Hoh_\etrm^1(X,\ZZ/2\ZZ)$ is trivial by assumption, hence $\calP$
is a direct sum of two line bundles by Proposition~\ref{PR:binary-decomposable}.
Finally, the algebra $A$ always has a symplectic involution by Example~\ref{EX:deg-two-Az-algebra}.
\end{proof}

It is easy to provide a projective scheme and an indecomposable rank 2
vector bundle, e.g., $X=\PP^2$ and $\calP=\Omega^1_{\PP^2}$.  We
now give an example of an integral affine scheme $X$ and a locally
free sheaf $\calP$ satisfying the conditions of Corollary
\ref{COR:no-orthogonal-involution}, thus giving rise to a split
Azumaya algebra of degree 2 admitting a symplectic involution but no
orthogonal involution over a domain.

\begin{example}\label{EX:no-orthogonal-involution}
Let $R= \CC[x,y,z,s,t,u]/(xs+yt+zu-1)$. The vector $v=(x,y,z)\in R^3$
is unimodular, hence $\calP=R^3/Rv$ is a projective $R$-module of rank
$2$. Indeed, the ring $R$ is $(A_{3,1})_\CC$ and the module $\calP$ is
$P_{3,1}$ in the notation of \cite{raynaud1968}. There is an affine
vector bundle torsor $\Spec R \to \mathbb{A}^3\setminus \{0\}$, which
on the level of coordinates is given by $(x,y,z,s,t,u) \mapsto
(x,y,z)$, and it follows that $\Pic(R) = \mathrm{CH}^1(\Spec R) = 0$,
see \cite[Thm.\ 1.9]{fulton1984}. By \cite[Cor.\
6.3]{raynaud1968}, the $R$-module $\calP$ is not free, and as a
consequence is indecomposable.

By Artin's Comparison theorem, $\Hoh_\etrm^1(R,\ZZ/2\ZZ)=0$, and thus
$A=\End(\calP)$ has a symplectic involution, but no orthogonal
involution by Corollary~\ref{COR:no-orthogonal-involution}.
\end{example}

\section{Non-Split Azumaya Algebras With Only Symplectic Involutions}

    We now show that any Azumaya algebra with involution over an affine scheme is a specialization of
    an involutary Azumaya algebra without zero divisors. Applying this
    to Example~\ref{EX:no-orthogonal-involution}, we obtain non-split Azumaya algebras of degree $2$
    admitting only symplectic involutions.


 We shall give two different constructions, both arising from orders in certain
 generic division algebras.
 The first construction is in fact a sequence of involutary Azumaya algebras that together specialize
 to any involutary Azumaya algebra. The centers of these algebras are regular, but their
 Krull dimension is very large.
 The second construction is not universal in the previous sense, but its center has smaller Krull dimension.

\medskip

 Henceforth, we shall restrict to Azumaya algebras over commutative rings. We write $A/R$ to
 denote that $A$ is an $R$-algebra.
 We remind the reader that if $A/R$ and $A'/R'$ are two Azumaya algebras
 of degree $n$, then any ring homomorphism $\phi:A\to A'$ is in fact
 a \emph{specialization}, meaning that $\phi(R)\subseteq R'$, and $A\otimes_\phi R'\cong A'$
 via $a\otimes r'\mapsto \phi(a)r'$ (see \cite[Cor.~2.9b]{Sa99}).


 We will also need Rowen's version of
 the Artin--Procesi Theorem.
 If $g(x_1,\dots,x_n)$ is a polynomial in non-commuting
 variables over $\ZZ$ and $A$ is a ring, then let
 $g(A)$ denote $\{g(a_1,\dots,a_n)\,|\, a_1,\dots,a_n\in A\}$. We further let
 $\mathrm{Id}(A)$ denote the \emph{multilinear polynomial identities} of $A$ over $\ZZ$
 (see \cite[Chp.~6]{Rowen88B}).
 The center of $A$ is denoted $Z(A)$.

 \begin{theorem}[Artin, Procesi, Rowen]\label{TH:artin-procesi}
 Let $A$ be a ring and let $g_n(x_1,\dots,x_{4n^2+1})$ denote the {Formanek central polynomial}
 (see \cite[Def.\ 6.1.21]{Rowen88B}).
 The following conditions are equivalent:
 \begin{enumerate}
 \item[(a)] $A$ is Azumaya of degree $n$ over $Z(A)$.
 \item[(b)] $\mathrm{Id}(\Mat_{n\times n}(\ZZ))\subseteq\mathrm{Id}(A)$,
 and there exists a multilinear polynomial identity of $\Mat_{(n-1)\times (n-1)}(\ZZ)$ that is not a polynomial
 identity of any nonzero image of $A$.
 \item[(c)] $\mathrm{Id}(\Mat_{n\times n}(\ZZ))\subseteq\mathrm{Id}(A)$, the polynomial $g_n$ is a central polynomial of $A$, and
 $1_A$ is in the additive group spanned by $g_n(A)$.
 \end{enumerate}
 \end{theorem}

 \begin{proof}
 See \cite[Thm.~6.1.35]{Rowen88B} or \cite[\S6]{Rowen74}, for instance.
 \end{proof}

\subsection{First Construction}

 This construction is inspired by Rowen's generic division algebras with involution
 \cite{Ro75}.


\medskip

 Fix an integral domain $\Omega$ with $2\in\Omega^\times$
 and fix $n>1$.
 Recall that $g_n$ denotes the Formanek central polynomial,
 and let $N=4n^2+1$ denote the number of variables of $g_n$. Theorem~\ref{TH:artin-procesi}
 implies that any Azumaya algebra $A/R$ of degree $n$ admits vectors $v_1,\dots,v_m\in A^{N}$
 such that $\sum_{i=1}^mg_n(v_i)=1_{A}$. We call the minimal possible such $m$ the \emph{Formanek number}
 of $A$ and denote it by $\mathrm{For}(A)$.

 For every $k\in\NN$, let $T=T(n,k)$ be the free commutative $\Omega$-algebra
 spanned by $\{x_{ij}^{(r)}\}_{i,j,r}$ where $1\leq i,j\leq n$ and $1\leq r\leq k$. We let $F$ denote
 the fraction field of $T$.
 Let $X_r$ denote the generic matrix $(x_{ij}^{(r)})\in\Mat_{n\times n}(T)$.
 Let $A_0(n,k)$ be the $\Omega$-subalgebra of $\Mat_{n\times n}(T)$ generated
 by $X_1,X_1^{\tr},\dots,X_{k},X_{k}^{\tr}$. If $n$ is even, let $B_0(n,k)$ be the $\Omega$-subalgebra
 of $\Mat_{n\times n}(T)$ generated by $X_1,X_1^{\syp},\dots,X_k,X_k^{\syp}$ (see Example~\ref{EX:involutions}).
 We alert the reader that the algebras $A_0(n,k)$ and $B_0(n,k)$ are not Azumaya in general.

 \begin{lemma}\label{LM:specialization}
 Let $(A,\sigma)$ be an Azumaya algebra of degree $n$ with an orthogonal (resp.\ symplectic)
 involution and such that $R:=Z(A)$ is an $\Omega$-algebra. Let $a_1,\dots,a_k\in A$. Then there is a homomorphism
 of $\Omega$-algebras with involution $\phi:(A_0(n,k),\tr)\to (A,\sigma)$
 (resp.\ $\phi:(B_0(n,k),\syp)\to (A,\sigma)$)
 such that $\phi(X_i)=a_i$ for all $1\leq i\leq k$.
 \end{lemma}

 \begin{proof}[Proof (compare with {\cite[Thm.~27(i)]{Ro75}})]
 We will treat the orthogonal case; the symplectic case is similar.

 The lemma is clear when $A=\Mat_{n\times n}(R)$ and $\sigma=\tr$,
 indeed, simply specialize $X_i$ to $a_i$
 for all $i$.
 For general $A$, choose an \'{e}tale $R$-algebra $S$ such that $(A',\sigma'):=(A\otimes_RS,\sigma\otimes_R\id_S)$
 is isomorphic to $(\Mat_{n\times n}(S),\tr)$ (see Proposition~\ref{PR:local-isom-of-algebras-with-inv}),
 and view $(A,\sigma)$ as an involutary subring of $(A',\sigma')$.
 Then there is $\phi:A_0(n,k)\to A'$
 with $\phi(X_i)=a_i$ and $\phi(X_i^{\tr})=\sigma'(a_i)=\sigma(a_i)$.
 Since $A_0(n,k)$ is generated as an $\Omega$-algebra by $X_1,\dots,X_k,X_1^{\tr},\dots,X_{k}^{\tr}$,
 we have $\mathrm{im}(\phi)\subseteq \Omega[a_1,\dots,a_k,\sigma(a_1),\dots,\sigma(a_k)]\subseteq A$.
 \end{proof}

 Suppose now that $k=Nm$ for $m\in\NN$ and let
 \[
 \omega=\sum_{i=1}^m g_n(X_{N(i-1)+1},\dots,X_{N(i-1)+N})\in A_0(n,k)\cap B_0(n,k)
 \]
 Since $g_n$ is a central polynomial of $\Mat_{n\times n}(T)$, the matrix $\omega$ is diagonal, and hence $\omega^\tr=\omega^\syp=\omega$.
 Furthermore, since $g_n$ is
 not a polynomial identity of $\Mat_{n\times n}(\Omega)$, we have $\omega\neq 0$ (because the matrices $X_1,\dots,X_{Nm}$
 can be specialized to any $n\times n$ matrix over $\Omega$).
 We define
 \[
 A(n,m):=A_0(n,Nm)[\omega^{-1}]\subseteq \Mat_{n\times n}(F)\ .
 \]
 The involution $\tr$ on $\Mat_{n\times n}(F)$ restricts to an involution on $A(n,m)$, which we also
 denote by $\tr$.
 In the same way, we define $B(n,m)$ for $n$ even by replacing ``$A_0$'' with ``$B_0$'' and ``$\tr$'' with ``$\syp$''.



 \begin{theorem}\label{TH:no-zero-divisors-II}
 Suppose $n>1$ is a power of $2$. Then
 $A(n,m)$ (resp.\ $B(n,m)$) is an Azumaya algebra of degree $n$ without zero divisors.
 Furthermore, $(A(n,m),\tr)$ (resp.\ $(B(n,m),\syp)$) specializes to any Azumaya algebra
 with an orthogonal (resp.\ symplectic) involution $(E,\sigma)$ satisfying $\deg E=n$ and
 $\mathrm{For}(E)\leq m$, and such that $Z(E)$ is an $\Omega$-algebra.
 \end{theorem}

 \begin{proof}
 Again, we prove only the orthogonal case.

 We use condition (c) of Theorem~\ref{TH:artin-procesi} to check that $A(n,m)$ is Azumaya.
 That $\mathrm{Id}(\Mat_{n\times n}(\ZZ))\subseteq \mathrm{Id}(A(n,m))$
 and $g_n$ is a central polynomial of $A(n,m)$ follow from the fact that $A(n,m)\subseteq \Mat_{n\times n}(F)$.
 Next, we have $\omega=\sum_{i=1}^m g_n(X_{N(i-1)+1},\dots,X_{N(i-1)+N})$,
 and hence
 \[1=\sum_{i=1}^m g_n(\omega^{-1}X_{N(i-1)+1},X_{N(i-1)+2},\dots,X_{N(i-1)+N})\]
 because $g_n$ is multilinear, so $1$ is in the additive group generated by $g_n(A(n,m))$, as required.

 That $A(n,m)$ has no zero divisors follows from \cite[Thm.~29]{Ro75}.

 To finish, let $(E,\sigma)$ be as in the theorem.
 Since $\mathrm{For}(E)\leq m$, there are vectors $v_1,\dots,v_m\in E^N$ such
 that $\sum_ig_n(v_i)=1_{E}$.
 Define $a_1,\dots,a_{Nm}\in E$ via $v_i=(a_{N(i-1)+1},\dots,a_{N(i-1)+N})$.
 By Lemma~\ref{LM:specialization}, there is a homomorphism
 $\phi:(A_0(n,Nm),\tr)\to (E,\sigma)$ such that $\phi(X_i)=a_i$ and $\phi(X_i^{\tr})=\sigma(a_i)$
 for all $1\leq i\leq Nm$.
 It follows that the element $\omega\in A_0(n,Nm)$ is mapped by $\phi$ to $1_E$.
 Thus, $\phi$ extends to a homomorphism of rings with involution $(A(n,m),\tr)\to (E,\sigma)$
 by setting $\phi(\omega^{-1})=1$.
 \end{proof}

 Write
 $
 Z(n,m)=Z(A(n,m))
 $
 and $W(n,m)=Z(B(n,m))$. We now show that the morphisms $\Spec Z(n,m)\to \Spec \Omega$
 and $\Spec W(n,m)\to \Spec \Omega$ are smooth when $\Omega$ is noetherian. As a result, $Z(n,m)$ and $W(n,m)$
 are regular when $\Omega$ is a field.

 \begin{lemma}\label{LM:noetherianity}
 Let $A/R$ be an Azumaya algebra. Suppose that there is a noetherian
 subring $R_0\subseteq R$ such that $A$ is finitely generated as an $R_0$-algebra.
 Then $R$ is finitely generated as an $R_0$-algebra. In particular, $A$ and $R$ are noetherian.
 \end{lemma}

 \begin{proof}
 Since $A$ is Azumaya, it is finitely generated as a module over $R$.
 The proposition therefore follows from a variant of the Artin--Tate Lemma (see \cite[Lem.~1]{MontSmall86},
 for instance).
 \end{proof}

 \begin{lemma}\label{LM:lift-mod-nil-ideal}
    Let $R$ be a commutative ring and let $I$ be an ideal in $R$ satisfying $I^2=0$.
    Let $G$ be a smooth affine group scheme over $R$. Then
    the map $\Hoh^1_{\et}(R,G)\to\Hoh^1_{\et}(R/I,G)$ is bijective.
 \end{lemma}

 \begin{proof}
    Write $R'=R/I$.

    By \cite[Exp.~viii, Thm.~1.1]{SGAiv-two}, the base change $\Spec R'\to \Spec R$ induces
    an equivalence between the \'{e}tale site of $\Spec R$ and the
    \'etale site of $\Spec R'$. It is therefore enough
    to prove that for any \'{e}tale cover
    $\Spec S\to \Spec R$, the map $\Hoh^1_{\et}(S/R,G)\to\Hoh^1_{\et}(S'/R',G)$  is an isomorphism,
    where $S'=S\otimes_R R'$.
    By \cite[Exp.~xxiv, Lem.~8.1.8]{SGAiii}, the map $\Hoh^1_{\fpqc}(S/R,G)\to \Hoh^1_{\fpqc}(S'/R',G)$ is an isomorphism
    (here we need $G$ to be smooth), and
    since $\Spec S\to \Spec R$ is \'{e}tale and $G\to\Spec R$ is smooth, we may replace the fpqc topology by the \'{e}tale topology, see the introduction  to \cite[Exp.~xxiv]{SGAiii}.
 \end{proof}

 \begin{proposition}\label{PR:smoothness}
    When $\Omega$ is noetherian, the morphisms $\Spec Z(n,m)\to \Spec \Omega$
    and $\Spec W(n,m)\to \Spec \Omega$ are smooth.
 \end{proposition}

 \begin{proof}
    We prove only  that $\Spec Z(n,m)\to \Spec \Omega$ is smooth.
    By Lemma~\ref{LM:noetherianity}, this morphism is finitely presented, and hence it is enough
    to show that $\Spec Z(n,m)\to \Spec \Omega$ is formally smooth.
    Let $S$ be a commutative $\Omega$-algebra, let $I$ be an ideal of $S$ with $I^2=0$, and let
    $\phi:Z(n,m)\to S/I$ be a homomorphism of $\Omega$-algebras. We show
    that $\phi$ can be lifted to a homomorphism $\phi':Z(n,m)\to S$.

    Let $A=A(n,m)\otimes_\phi (S/I)$ and let $\sigma$ be the involution induced by $\tr$ on $A$.  Then $(A,\sigma)$ is
    an Azumaya algebra with an orthogonal involution over $S/I$.  The smoothness of $\PGO_n\to \Spec \ZZ[\frac{1}{2}]$
    and Lemma~\ref{LM:lift-mod-nil-ideal} (see also Section~\ref{subsection:cohomology}) imply 
    that there is an Azumaya $S$-algebra with involution $(A',\sigma')$ such
    that $(A,\sigma)\cong (A'\otimes_S(S/I),\sigma'\otimes_S\id_{S/I})$.  We identify $A$ with $A'/A'I$.  For all $1\leq
    i\leq Nm$, write $a_i=\phi(X_i)$ and choose some $a'_i\in A'$ whose image in $A$ is $a_i$.  By
    Lemma~\ref{LM:specialization}, there is a morphism of $\Omega$-algebras with involution $\phi':(A_0(n,Nm),\tr)\to
    (A',\sigma')$ such that $\phi'(X_i)=a'_i$. The morphism $\phi'$ extends uniquely to $A(n,m)$ provided
    $\phi'(\omega)\in {A'}^\times$.  This holds because $\phi'(\omega)+A'I=\phi(\omega)\in A^\times$ (since
    $\phi(\omega)^{-1}=\phi(\omega^{-1})$) and $A'I$ is nilpotent. We have shown $\phi' : A(n,m) \to A'$ is a
    homomorphism of Azumaya algebras, and therefore it restricts to a homomorphism on the centers $\phi': Z(n,m) \to S$,
    which is the lift we required.
 \end{proof}

 \begin{corollary}
    There exists an affine regular $\CC$-algebra $R$ and an Azumaya $R$-algebra $A$ of degree
    $2$ without zero divisors such that $A$ has no orthogonal involutions.
 \end{corollary}

 \begin{proof}
 Let  $m_0$ to be the Formanek number of the Azumaya algebra constructed in Example~\ref{EX:no-orthogonal-involution}.
 Take $\Omega=\CC$ and let $R=W(n,m_0)$ and $A=B(n,m_0)$.
 \end{proof}

\subsection{Second Construction}

 This construction uses versal torsors in the sense of \cite{ReichDun15}.
 It lacks the universal character of the first construction and it may have a non-regular center,
 but the Krull dimension of the center is much smaller and can be effectively estimated.

 \medskip

    We start by showing that weakly versal $\PGO_{2n}$-torsors and $\PGSp_{2n}$-torsors satisfy a slightly stronger
    version of weak versality.

    Let $K$ be a field and let $G$ be a smooth affine group scheme over $K$. Recall that a weakly versal
    $G$-torsor is a $K$-scheme $X$ together with a $G$-torsor $T\to X$ such for every field extension
    $K'/K$ with $K'$ infinite and every $G$-torsor $T'\to \Spec K'$ there exists a $K$-morphism $i:\Spec K'\to X$ such
    that $T'\cong i^*T$.
    The torsor $T\to X$ is called versal if, for every open immersion $j:U\to X$,
    the restriction $j^*T\to U$ is weakly versal.
    The minimal possible dimension of a base scheme of a versal torsor is equal to the essential
    dimension of $G$, denoted $\ed_K(G)$. These versal torsors of minimal dimension can be chosen to be
    integral and affine over $K$. See \cite{BerhFavi03}, \cite{ReichDun15} for further details and proofs.



 \begin{lemma}\label{LM:inverse-image}
 Let $L$ be a field and let $R$ be a commutative ring.
 Let $E/L$ and $A/R$ be Azumaya algebras of degree $n$.
 Suppose that there is a subring $E_0\subseteq E$ and a surjective
 ring homomorphism $\phi:E_0\to A$. Then there exist
 elements $s\in E_0\cap {L}^\times$ and $x_1,\dots,x_m\in E_0$ such that:
 \begin{enumerate}
 \item[(i)] $\phi$ extends uniquely
 to a ring homomorphism $E_0[s^{-1}]\to A$.
 \item[(ii)] Any subring $E_1\subseteq E$ containing $s^{-1}$ and
 $x_1,\dots,x_m$ is Azumaya of degree $n$ over its center.
 \end{enumerate}
 In particular, $E_0[s^{-1}]$ is an Azumaya algebra
 that specializes to $A$ via $\phi$.
 \end{lemma}

 \begin{proof}
 Let $N=4n^2+1$ denote the number of variables of the Formanek central polynomial $g_n$.
 Since $A$ is Azumaya of degree $n$, Theorem~\ref{TH:artin-procesi}
 implies that there are vectors $v_1,\dots,v_t\in A^N$
 such that $\sum_{i=1}^tg_n(v_i)=1$. Choose $u_1,\dots,u_t\in E_0^N$ such
 that $\phi_0^N(u_i)=v_i$ for all $1\leq i\leq t$.
 Write
 $m =tN$ and define $x_1,\dots,x_m$ via $u_i=(x_{(i-1)N+1},\dots,x_{(i-1)N+N})$.
 In addition, let $s=\sum_{i=1}^tg_n(u_i)$. Observe that $s\in Z(E)=L$
 because $g_n$ is a central polynomial, and $s\neq 0$ because $\phi(s)=\phi(\sum_ig_n(u_i))=\sum_ig_n(v_i)=1_{A}$.
 We claim that $s$ and $x_1,\dots,x_m$ satisfy (i) and (ii).
 Indeed, extend $\phi$ to $E_0[s^{-1}]$ by defining $\phi(s^{-1})=1_{A}$. 
 That any $E_1$ as in (ii) is Azumaya is similar to the proof of Theorem~\ref{TH:no-zero-divisors-II}.
 \end{proof}

 \begin{proposition}\label{PR:wierd-versality}
    Let $K'/K$ be a field extension with $K'$ infinite, and let $G$ be $\PGO_{n}$ or $\PGSp_{n}$.
    Let $T\to X$ be a weakly versal $G$-torsor  over $K$ such that $X$ is integral  and affine over $\Spec K$,
    and let $X'=X\times_{\Spec K}\Spec K'$ and $T'=T\times_{\Spec K}\Spec K'$.
    Let $R$ be a subring of $K'$ and let $U\to \Spec R$ be a $G$-torsor.
    Then there are $G$-torsors $T'_0\to X'_0$, $U'\to Y$ and morphisms as illustrated
    \[
    \xymatrix{
    U \ar[d] & U'\ar[d] & T'_0\ar[d] & T' \ar[d]\\
    \Spec R \ar[r]^f & Y & X'_0 \ar[l]_j \ar[r]^i & X'
    }
    \]
    such that $i$ is an open immersion, $j$ is dominant, and there are isomorphisms
    of $G$-torsors
    \[U\cong f^*U',\qquad  j^*U'\cong T'_0\cong i^*T'\ .\]
    If $R$ contains a noetherian  subring $R_0$, then $Y$ may be chosen to be of finite type
    over $\Spec R_0$ and $f$ may be chosen to be an $R_0$-morphism.
 \end{proposition}

 \begin{proof}
    We treat only  the case $G=\PGO_{n}$. The other case is similar.
    If $R_0$ is not specified, take it to be the image of  $\ZZ[\frac{1}{2}]$ in $R$.
    Since weak versality is preserved under base change, we may replace $T\to X$ by
    $T'\to X'$ and assume that $K'=K$.

    Let $(B,\tau)$ be the involutary Azumaya algebra corresponding to $T\to X$
    and let $(A,\sigma)$ be the involutary Azumaya algebra  corresponding to $U\to \Spec R$.
    Since $X$ is weakly versal, there is a $K$-section $g: \Spec K\to X$
    specializing $(B,\tau)$ to $(A_{K},\sigma_{K}):=(A\otimes_R {K},\sigma\otimes_R\id_{K})$.
    Denote the induced ring homomorphism by $\phi:B\to A_{K}$. The map $\phi$ is onto
    because it is $K$-linear and any epimorphic image of $B$ is an Azumaya algebra of degree $n$.
    View $A$ as a subring of $A_K$
    and let $E_0=\phi^{-1}(A)$. Since $\phi$ is a homomorphism of rings with involution,
    we have $\tau(E_0)\subseteq E_0$.
    Let $\xi$ be the generic point of $X$ and let $E/L$ be the Azumaya algebra induced by
    the $\PGO_{n}$-torsor $T_{\xi}$. The involution $\tau:B\to B$ extends to an involution on $E$.
    We now apply Lemma~\ref{LM:inverse-image}
    to get elements $x_1,\dots,x_m\in E_0$ and $s\in Z(E_0)\cap L^\times$.

    Let $E_1=R_0[s^{-1},x_1,\dots,x_m,\tau(x_1),\dots,\tau(x_m)]$.
    Then $E_1$ is Azumaya of degree $n$ over its center, and $\phi$ induces a specialization
    of Azumaya algebras from $E_1$ to $A$. This specialization is compatible with the involutions $\tau$
    and $\sigma$, hence $(E_1,\tau|_{E_1})$ specializes to $(A,\sigma)$.
    We therefore take $Y=\Spec Z(E_1)$ and $X'_0=\Spec Z(B)[s^{-1}]$, and let $U$ and $T'_0$
    be the $\PGO_{n}$-torsors corresponding to
    $(E_1,\tau)$ and $(B[s^{-1}],\tau)$, respectively.
    The morphism $i$ corresponds to the inclusion $Z(B)\subseteq Z(B)[s^{-1}]$,
    the morphism $j$ corresponds to the inclusion $Z(E_1)\subseteq Z(E_0)[s^{-1}]\subseteq Z(B)[s^{-1}]$,
    and
    the morphism $f$ corresponds to $\phi:Z(E_1)\to Z(A)=R$.
    When $R_0$ is noetherian, $Z(E_1)$ is affine over $R_0$ by Lemma~\ref{LM:noetherianity}.
 \end{proof}

    Proposition~\ref{PR:wierd-versality} also
    holds for $\PGL_n$-torsors; the proof is similar. We do not know whether the proposition holds
    for torsors of general smooth affine group schemes.

 \medskip

    We now use Proposition~\ref{PR:wierd-versality} to show that any Azumaya algebra with involution
    $(A,\sigma)$ over a domain $R$ is a specialization of an involutary Azumaya algebra without zero divisors.

 \begin{lemma}\label{LM:no-zero-divisors}
        Let $K$ be a field and let $T\to X$ be a versal $\PGO_{n}$-torsor (resp.\ $\PGSp_{n}$-torsor).
        Suppose that $X$ is integral and affine, $X=\Spec R$,
        and let $(A,\sigma)$ denote the Azumaya algebra with involution corresponding to $T \to X$.
        If $n$
        is a power of $2$, then $A$ has no zero divisors.
    \end{lemma}

    \begin{proof}
        It is well-known that there exists a field extension $K'/K$ such that $K'$ is infinite
        and a central simple $K'$-algebra
        with an orthogonal (resp.\ symplectic) involution $(D,\tau)$ of degree $n$ such that $D$ is a division ring;
        see  \cite[Thm.~29]{Ro75}, for instance.

        Suppose that $a,b\in A$ satisfy $ab=0$. View $A$ as a sheaf of algebras over $X$ and let $Z_a$ and $Z_b$ be the
        vanishing loci of $a$ and $b$, respectively. Let $U=X\setminus (Z_a\cup Z_b)$ and assume for the sake of
        contradiction that $U$ is nonempty. Shrinking $U$ if necessary, we may assume $U=\Spec R[r^{-1}]$ 
        for some $0\neq r\in R$.  By the versality of $T\to X$, the algebra $(A[r^{-1}],\sigma)$ specializes
        to $(D,\tau)$ over $K$, and by the construction of $U$, the images of $a$ and $b$ in $D$ are nonzero. This means
        $D$ has zero divisors, a contradiction. Consequently $U$ must be empty and since $X$ is irreducible, either
        $Z_a=X$ or $Z_b=X$. This implies that either $a=0$ or $b=0$, since $X$ is reduced.
    \end{proof}

 \begin{theorem}\label{TH:no-zero-divisors}
 Let $(A,\sigma)$ be an Azumaya algebra of degree $2^n$ with an orthogonal
 (resp.\ symplectic) involution  over a domain $R$.
 Then there exists an Azumaya algebra of degree $2^n$ with an orthogonal (resp.\ symplectic) involution
 $(B,\tau)$ such that $B$ has no zero divisors and $(B,\tau)$ specializes to $(A,\sigma)$.
 If $R$ is an affine algebra over a field $k$, then $(B,\tau)$ can be chosen such that
 $Z(B)$ is affine over $k$ and
 \[\dim Z(B)\leq \dim R+\ed_K(G)\]
 where $G=\PGO_{2^n}$  (resp.\ $G=\PGSp_{2^n}$) and $K$ is the fraction field of $R$.
 \end{theorem}

 \begin{proof}
    We prove only the orthogonal case.

    Let $U\to \Spec R$ be the torsor corresponding to $(A,\sigma)$.
    Choose a versal $\PGO_{2^n}$-torsor $T\to X$ over $K$ such that $X$ is an affine, integral $K$-scheme
    and $\dim X=\ed_K(\PGO_{2^n})$. We now apply Proposition~\ref{PR:wierd-versality} with $K'=K$
    (and $R_0=k$, if necessary)
    to obtain $U'\to Y$, $T'\to X'_0$, $f$, $i$, $j$ as in the proposition.

    Let $(B,\tau)$ and $(E,\theta)$ be the Azumaya algebras with involution corresponding to $U'\to Y$,
    and $T\to X$,  respectively.
    By Lemma~\ref{LM:no-zero-divisors}, $E$ has no zero divisors. Since $i$ is an open immersion and $j$ is dominant,
    this means   $B$ has no zero divisors.
    Finally, if $R$ is affine over $k$,
    then $\dim Z(B)=\trdeg_kK(Y)\leq \trdeg_k K(X)=\trdeg_k K+\trdeg_K K(X)=\dim R+\ed_K(\PGO_{2^n})$.
 \end{proof}

 We note that Theorem~\ref{TH:no-zero-divisors}
 gives an Azumaya algebra whose center is \textit{a priori} not regular.





 \begin{corollary}
 There exists an Azumaya algebra $B/S$ of degree $2$ without zero divisors
 that admits only symplectic involutions. The ring $S$ can be taken to be an affine
 $\CC$-algebra with $\dim S\leq 7$.
 \end{corollary}

 \begin{proof}
 This follows from Example~\ref{EX:no-orthogonal-involution} and Theorem~\ref{TH:no-zero-divisors} since
 $\ed_{\CC}(\PGSp_2)\leq2$.
 \end{proof}




\end{document}